\documentclass[a4paper, 11pt, reqno]{amsart}
\usepackage[english]{babel}
\usepackage[utf8]{inputenc}
\usepackage{amsthm, amsmath, amssymb, amsrefs, enumerate, enumitem, mathtools, mathrsfs, microtype, lmodern, hyperref}

\usepackage[top=1.5in, bottom=1.5in, left=0.9in, right=0.9in]{geometry}
\binoppenalty=\maxdimen
\relpenalty=\maxdimen

\theoremstyle{plain}
\newtheorem{theorem}{Theorem}[section]
\newtheorem{lemma}[theorem]{Lemma}
\newtheorem{corollary}[theorem]{Corollary}
\newtheorem{proposition}[theorem]{Proposition}

\theoremstyle{definition}
\newtheorem*{definition}{Definition}

\theoremstyle{remark}
\newtheorem*{remark}{Remark}
\newtheorem*{remarks}{Remarks}

\makeatletter
\let\@@pmod\pmod
\DeclareRobustCommand{\pmod}{\@ifstar\@pmods\@@pmod}
\def\@pmods#1{\mkern4mu({\operator@font mod}\mkern 6mu#1)}
\makeatother

\newcommand{\C}{\mathbb{C}}
\renewcommand{\H}{\mathbb{H}}
\newcommand{\Z}{\mathbb{Z}}
\newcommand{\Q}{\mathbb{Q}}
\newcommand{\N}{\mathbb{N}}
\newcommand{\R}{\mathbb{R}}
\newcommand{\slz}{{\text {\rm SL}}_2(\mathbb{Z})}
\newcommand{\e}{\mathfrak{e}}
\newcommand{\re}{\textnormal{Re}}

\newcommand{\vt}[1]{\left\lvert #1 \right\rvert}
\newcommand{\sgn}{\operatorname{sgn}}

\DeclareMathOperator{\Gr}{Gr}
\DeclareMathOperator{\reg}{reg}
\DeclareMathOperator{\tr}{tr}
\newcommand{\CT}{\mathrm{CT}}
\newcommand{\sfr}{\mathfrak{s}}
\newcommand{\Fc}{\mathcal{F}}
\newcommand{\Gc}{\mathcal{G}}
\renewcommand{\Mc}{\mathcal{M}}
\newcommand{\Hc}{\mathcal{H}}
\newcommand{\Hs}{\mathscr{H}}
\newcommand{\Qc}{\mathcal{Q}}
\newcommand{\Fb}{\mathbb{F}}

\title[Eichler--Selberg type relations for negative weights]{Local Maa{\ss} forms and Eichler--Selberg type relations for negative weight vector-valued mock modular forms}
\author{Joshua Males}
\address{450 Machray Hall, Department of Mathematics, University of Manitoba, Winnipeg,
	Canada}
\email{joshua.males@umanitoba.ca}
\author{Andreas Mono}
\address{Department of Mathematics and Computer Science, Division of Mathematics, University of Cologne, Weyertal 86-90, 50931 Cologne, Germany}
\email{amono@math.uni-koeln.de}

\begin{document}



\thanks{The research conducted by the first author for this paper is supported by the Pacific Institute for the Mathematical Sciences (PIMS). The research and findings may not reflect those of the Institute. The second author was supported by the CRC/TRR 191 ``Symplectic Structures in Geometry, Algebra and Dynamics'', funded by the DFG (project number 281071066).}

\maketitle

\section*{Modifications to the published version below}

\begin{enumerate}
\item Throughout the paper, we add the assumption that our homogeneous polynomial $p$ inside the Siegel theta function is equal to $1$. Otherwise, the Siegel theta function might not split into a positive definite and a negative definite part in general. In particular, one has to add additional assumptions on the isometry $\psi$ as well as on the polynomial $p$ to obtain such a splitting, see \cite{zemel2}*{Lemma 2.2} and the discussion preceeding it. Furthermore, finding a preimage of $\Theta_P$ under the shadow operator $\xi$ might not be guaranteed for non-constant polynomials $p$, and our Proposition 4.1 is wrong if $p \neq 1$ since the Laplacian depends on the given polynomial, see \cite{zemel1}*{Proposition 2.5} for the correct version.

\item In Theorem 1.2, we need to specialize the signature of the lattice $L$ to $(2,1)$ instead of $(2,s)$. This is necessary, because the nature of the singularities of the lift is different in higher dimensions, see the recent preprint \cite{adks}. In particular, the first condition in our definition of a local Maa{\ss} form on page $389$ simplifies to the usual scalar-valued modularity condition, see Bringmann, Kane and Viazovska \cite{BKV}*{Subsection 2.4} as well. In general, the Siegel theta function is invariant under the discriminant kernel of $\mathrm{O}(L)$ as a function of $Z \in \mathrm{Gr}(L)$, see \cite{bruinier2004borcherds}*{p.\ 40}. In the case of signature $(2,1)$, we have $\mathrm{Gr}(L) \cong \mathbb{H}$, and choosing a particular lattice of that signature leads to further identifications, which in turn yield the framework of \cite{BKV}. This is described in Section $5$ of our paper below.
\end{enumerate}

\section*{Acknowledgement}
The authors would like to thank Paul Kiefer for pointing out the aforementioned errors and for many helpful conversations. Moreover, we thank the referee for a helpful comment.

\newpage

\begin{center}
\large \uppercase{\textbf{Published version}}
\end{center}

\smallskip

\begin{abstract}
By comparing two different evaluations of a modified (\`{a} la Borcherds) higher Siegel theta lift on even lattices of signature $(r,s)$, we prove Eichler--Selberg type relations for a wide class of negative weight vector-valued mock modular forms. In doing so, we detail several properties of the lift, as well as showing that it produces an infinite family of local (and locally harmonic) Maa{\ss} forms on Grassmanians in certain signatures.
\end{abstract}

\bigskip

\section{Introduction}

Theta lifts have a storied history in the literature, receiving a vast amount of attention in the past few decades with applications throughout mathematics. In this paper, we are concerned with generalizations of the Siegel theta lift originally studied by Borcherds in the celebrated paper \cite{bor}. The classical Siegel lift maps half-integral weight modular forms to those of integral weight, and has seen a wide number of important applications. For example, in arithmetic geometry \cites{BrZe,ES}, deep results in number theory \cite{BO2}, fundamental work of Bruinier and Funke \cite{bruinierfunke2004}, among many others.

More recently, Bruinier and Schwagenscheidt \cite{BS} investigated the Siegel theta lift on Lorentzian lattices (that is, even lattices of signature $(1,n)$), and in doing so provided a construction of recurrence relations for mock modular forms of weight $\frac{3}{2}$, as well as commenting as to how one could provide a similar structure for those of weight $\frac{1}{2}$, thereby including Ramanujan's classical mock theta functions.

In the last few years, several authors have also considered so-called ``higher'' Siegel theta lifts of the shape ($k \coloneqq \frac{1-n}{2}, j \in \N_0$)
\begin{align*}
\int_{\Fc}^{\reg} \left\langle R_{k-2j}^j f, \overline{\Theta_L(\tau,z)} \right\rangle v^k d\mu(\tau)
\end{align*}
where $R_{\kappa}^n \coloneqq R_{n-2} \circ R_{n-4} \circ \dots \circ R_{\kappa}$ is an iterated version of the Maa{\ss}  raising operator $R_\kappa \coloneqq 2i\frac{\partial}{\partial\tau} + \frac{\kappa}{v}$, $f$ is weight $k-2j$ harmonic Maa{\ss} form, and $\Theta_L$ is the standard Siegel theta function associated to an even lattice $L$ of signature $(1,n)$. Here and throughout, $\tau = u+iv \in \H$ and $z \in \Gr(L)$, the Grassmanian of $L$. Furthermore, $\langle \cdot, \cdot \rangle$ denotes the natural bilinear pairing. For example, they were considered by Bruinier and Ono (for $k=0, j=1$) in the influential work \cite{BO}, by Bruinier, Ehlen, and Yang in in the breakthrough paper \cite{bruinier2020greens} in relation to the Gross-Zagier conjecture, and by Alfes-Neumann, Bringmann, Schwagenscheidt and the first author in \cite{ANBMS} for $n=2$ and generic $j$. 

 In \cite{Mer}, Mertens investigated the classical Hurwitz class numbers, denoted by $H(n)$ for $n \in \N$. Using techniques in (scalar-valued) mock modular forms, he gave an infinite family of class number relations for odd $n$, two of which are
\begin{align} \label{Mertens rel}
\sum_{s \in \Z} H(n-s^2) + \lambda_1(n) = \frac{1}{3} \sigma_1(n), \qquad \sum_{s\in \Z}(4s^2 -n) H(n-s^2) + \lambda_3(n) = 0,
\end{align}
where $\lambda_k(n) = \frac{1}{2} \sum_{d \mid n} \min(d,\frac{n}{d})^k$ and $\sigma_k$ is the usual $k$-th power divisor function. Because of their close similarity to the classical formula of Kronecker \cite{Kro} and Hurwitz \cites{Hur1,Hur2}
\begin{align*}
\sum_{s \in \Z} H(n-s^2) - 2\lambda_1(n) = 2\sigma_1(n),
\end{align*}
and those arising from the Eichler--Selberg trace formula, Mertens referred to the relationships \eqref{Mertens rel} as \textit{Eichler--Selberg type relations}. More generally, let $[\cdot,\cdot]_\nu$ denote the $\nu$-th Rankin--Cohen bracket (see Section \ref{Sec: prelims}). In general, the Rankin-Cohen bracket $[f,g]$ is a mixed mock modular form of degree $\nu$. It is of inherent interest to determine its natural completion, say $\Lambda$, to a holomorphic modular form. Then following Mertens \cite{Mer1}, we say that a (mock-) modular form $f$ satisfies an Eichler--Selberg type relation if there exists some holomorphic modular form $g$ and some form $\Lambda$ such that
\begin{align*}
\left[f,g\right]_\nu +\Lambda
\end{align*}
is a holomorphic modular form. In the influential paper \cite{Mer1}, Mertens showed the beautiful result that all mock-modular forms of weight $\frac{3}{2}$ with holomorphic shadow satisfy Eichler--Selberg type relations, using the powerful theory of holomorphic projection and the Serre-Stark theorem stating that unary theta series form a basis for the spaces of holomorphic modular forms of the dual weight $\frac{1}{2}$.\footnote{Mertens also provides results for mock theta functions in weight $\frac{1}{2}$, but since there is no analogue of Serre-Stark in the dual weight $\frac{3}{2}$ this is a real restriction.} In particular, Mertens explicitly describes the form $\Lambda$ which completes the Rankin--Cohen brackets.

Following previous examples, to demonstrate the statement, let $\Hc$ denote the generating function of Hurwitz class numbers, let $\vartheta = \sum_{n \in \Z} q^{n^2}$, where $\tau \in \H$, and $q^n = \mathrm{e}^{2\pi i n \tau}$ throughout. Then Mertens' results show that \cite{Mer1}*{pp. 377} 
\begin{align*}
\left[\Hc,\vartheta \right]_\nu + 2^{-2\nu-1}\binom{2\nu}{\nu} \left( \sum_{r\geq 1} 2\sum_{\substack{m^2-n^2=r\\ m,n\geq 1}} (m-n)^{2\nu-1}\ q^r + \sum_{r\geq 1} r^{2\nu+1}q^r\right)
\end{align*}
is a holomorphic modular form of weight $2\nu+2$ for all $\nu \geq 1$, and a quasimodular form of weight $2$ if $\nu=0$.

In \cite{Ma}, the first author combined techniques of \cites{ANBMS,BS} during a further investigation of the higher Siegel lift on Lorentzian lattices. This lift was shown to be central in producing certain Eichler--Selberg type relations in the vector-valued case, providing an analogue of the scalar-valued weight $\frac{3}{2}$ case of Mertens. We remark that the shape of the form $\Lambda$ in the case of signature $(1,1)$ is very close to that of Mertens (see \cite{Ma}*{Theorem 1.1}), though we do not recall it here to save on complicated definitions in the introduction. 

In the current paper, we develop the theory for even generic signature $(r,s)$ lattices $L$ and more general modified Siegel theta functions as in Borcherds \cite{bor}, and consider the lift
\begin{align*}
\Psi^{\reg}_j \left(f, z\right)
\coloneqq
\int_{\Fc}^{\reg} \left\langle R_{k-2j}^{j}(f)(\tau) , \overline{\Theta_L(\tau,\psi,p)} \right\rangle v^{k} d\mu(\tau),
\end{align*}
where $\Theta_L$ is a modified Siegel theta function as in Borcherds \cite{bor}, essentially obtained by including a certain polynomial $p$ in the summand of the usual vector-valued Siegel theta function. We require $p$ to be homogenous and spherical of degree $d^+ \in \N_0$ in the first $r$ variables, and $d^- \in \N_0$ in the last $s$ variables (see Section \ref{Sec: theta functions} for precise definitions). Here, $\psi$ is an isometry which in turn defines $z$ - see Section \ref{Sec: theta functions}. Modifying the theta function in this way preserves modular properties of $\Theta_L$, while allowing us to obtain different weights of output functions. Furthermore, since the case $j=0$ is well-understood in the literature, we assume throughout that $j>0$. We remark that the signature $(1,2)$ with $j=0$ case has also been studied in \cites{Craw,CF}.

In particular, we evaluate the higher lift in the now-standard ways of unfolding in Theorem \ref{cor: theta lift as 2F1}, as well as recognizing it as a constant term in the Fourier expansion of the Rankin--Cohen bracket of a holomorphic modular form and a theta function (up to a boundary integral that vanishes for a certain class of input functions) in Theorem \ref{Theorem: theta lift at CM points}. For the second of these theorems, we use that at special points $w$, one may define positve- and negative-definite sublattices $P \coloneqq L \cap w$ and $N \coloneqq L \cap w^\perp$. In the simplest case, which we assume for the introduction, we have that $L = P \oplus N$. Then the theta series splits as $\Theta_{L} = \Theta_{P} \otimes \Theta_{N}$, where $\Theta_P$ is a positive definite theta series, and $\Theta_N$ a negative definite one. Then we let $\Gc_P^+$ be the holomorphic part of a preimage of $\Theta_P$ under $\xi_\kappa \coloneqq 2iv^\kappa \overline{\frac{\partial}{\partial {\overline{\tau}}}}$. For the sake of simplicity, we assume that $\Gc_P^+ +g$ in the statement of Theorem \ref{Thm: main} is bounded at $i\infty$ in the introduction; we overcome this assumption in Theorem \ref{Thm: ES} and offer a precise relation there.
Following the ideas of \cite{Ma}, by comparing these two evaluations of our lift and invoking Serre duality, we obtain the following theorem.

\begin{theorem}\label{Thm: main}
Let $L$ be an even lattice of signature $(r,s)$, with associated Weil representation $\rho_L$. Let $g$ be any holomorphic vector-valued modular form of weight $2-\left(\frac{r}{2}+d^+\right)$ for $\rho_L$. Suppose that $\Gc_P^+ +g$ is bounded at $i\infty$. Then $\Gc_P^+ + g$ satisfies an explicit Eichler--Selberg type relation. In particular, the form $\Lambda$ is explicitly determined.
\end{theorem}

The concept of so-called locally harmonic Maa{\ss} forms was introduced by Bringmann, Kane, and Kohnen in \cite{BKK}. These are functions that behave like classical harmonic Maa{\ss} forms, except for an exceptional set of density zero, where they have jump singularities. Since their inception, locally harmonic Maa{\ss} forms have seen applications throughout number theory, for example in relation to central values of $L$-functions of elliptic curves \cite{Ehl}, as well as traces of cycle integrals and periods of meromorphic modular forms \cites{ANBMS,LoeSch} among many others. Examples of such locally harmonic Maa{\ss} forms are usually achieved in the literature via similar theta lift machinery to that studied here. In addition to the direction of Theorem \ref{Thm: main}, we also discuss the action of the Laplace--Beltrami operator on the lift $\Psi^{\reg}_j$ in Theorem \ref{Thm: eigenval}. In doing so, we prove the following theorem, thereby providing an infinite family of local Maa{\ss} forms (and locally harmonic Maa{\ss} forms) in signatures $(2,s)$. To state the result, we let $F_{m,k-2j,\sfr}$ be a Maa{\ss}-Poincar\'{e} series as defined in Section \ref{Sec: MP}.

\begin{theorem}\label{Thm: intro lhmf}
Let $L$ be an even isotropic lattice of signature $(2,s)$. Then the lift $ \Psi^{\reg}_j(F_{m,k-2j,\sfr},z)$ is a local Maa{\ss} form on $\Gr(L)$ with eigenvalue $(\sfr-\frac{k}{2})(1-\sfr-\frac{k}{2})$ under the Laplace-Beltrami operator.
\end{theorem} 

We provide an example of an input function to our lift. To this end, we specialize our setting to signature $(1,2)$, in which case vector-valued modular forms can be identified with the usual scalar-valued framework on the complex upper half plane, and in particular $\Gr(L) \cong \H$. (We explain the required choices in Section \ref{Sec: CES}.) In $1975$, Cohen \cite{coh75} defined the generalized class numbers
\begin{align*}
H(\ell-1,|D|) \coloneqq \begin{cases}
0 & \text{ if } D \neq 0,1 \pmod*{4}, \\
\zeta(3-2\ell) & \text{ if } D = 0, \\
L\left(2-\ell,\left(\frac{D_0}{\cdot}\right)\right) \sum_{d \mid j} \mu(d)\left(\frac{D_0}{d}\right)d^{\ell-2}\sigma_{2\ell-3}\left(\frac{j}{d}\right) & \text{ else},
\end{cases}
\end{align*}
where $D = D_0 j^2$, as well as their generating functions
\begin{align*}
\Hc_{\ell}(\tau) &\coloneqq \sum_{n \geq 0} H(\ell,n)q^n, \qquad \ell \in \N\setminus\{1\}.
\end{align*}
Here, $\zeta$ refers to the Riemann zeta function, $L(s,\chi)$ to the Dirichlet $L$-function twisted by a Dirichlet character $\chi$, and $\mu$ is the Möbius function. The functions $\Hc_{\ell}$ are known as Cohen--Eisenstein series today, and can be viewed as half integral weight analogues of the classical integral weight Eisenstein series. Note that the numbers $H(2,n)$ are precisely the Hurwitz class numbers introduced above, and $\Hc_2 = \Hc$. Cohen proved that $\Hc_{\ell} \in M_{\ell-\frac{1}{2}}(\Gamma_0(4))$, the space of scalar-valued modular forms of weight $\frac{1}{2}$ on the usual congruence subgroup $\Gamma_0(4)$, and the coefficients satisfy Kohnen's plus space condition by definition. (See \cite{thebook}*{eq.\ (2.13) to (2.15), Corollary 2.25} for more details on this.) 

However, evaluating our lift requires negative weight, and a non-constant principal part of the input function. To overcome both obstructions, we let 
\begin{align*}
f_{-2\ell,N}(\tau) = q^{-N} + \sum_{n > m} c_{-2\ell}(N,n) q^n, \quad N \geq -m, \quad m \coloneqq \begin{cases}
\lfloor \frac{-2\ell}{12}\rfloor-1 & \text{if } -2\ell \equiv 2\pmod*{12}, \\
\lfloor \frac{-2\ell}{12}\rfloor & \text{else},
\end{cases}
\end{align*}
be the unique weakly holomorphic modular form of weight $-2\ell$ for $ \slz$ with such a Fourier expansion, an explicit description of $f_{-2\ell,N}$ was given by Duke, Jenkins \cite{duje}, and by Duke, Imamo\={g}lu, T\'{o}th \cite{duimto10}*{Theorem 1}. Our machinery now enables us to obtain Eichler--Selberg type relations for the weakly holomorphic function $f_{-2\ell,N}(\tau)\Hc_{\ell}(\tau)$ along the lines of \cite{coh75}*{Section $6$}, as well as the following variant of Theorem \ref{Thm: intro lhmf}.

\begin{theorem} \label{Cor: intro lhmf}
The lift $\Psi^{\reg}_j\left(f_{-2\ell,N}\Hc_{\ell},z\right)$ is a local Maa{\ss} form on $\H$ for every $j \in \N$, $\ell \in \N\setminus\{1\}$, and $-m \leq N \in \N$ with exceptional set given by the net of Heegner geodesics
\begin{align*}
\bigcup_{D = 1}^N \left\{z = x+iy \in \H \colon \exists a,b,c \in \Z, \ b^2-4ac=D, \ a\vt{z}^2+bx+c = 0\right\}.
\end{align*}
\end{theorem}

\begin{remarks}
\
\begin{enumerate}
\item Theorem \ref{Cor: intro lhmf} generalizes immediately to any weakly holomorphic modular form $g$. The exceptional set is given by the union of geodesics of discriminant $D>0$, for which the coefficient of $g$ at $q^{-D}$ is non-zero. 
\item Recently, Wagner \cite{Wag} constructed a pullback of $\Hc_{\ell}$ under the $\xi$-operator, namely a harmonic Maa{\ss} form $\Hs_{\ell}$ of weight $-\ell+\frac{1}{2}$ on $\Gamma_0(4)$ that satisfies $\xi_{\frac{1}{2}-\ell}\Hs_{\ell} = \Hc_{\ell+2}$. An explicit definition of $\Hs_{\ell}$ can be found in \cite{Wag}*{eq.\ (1.5), (1.6)}. However, $\Hs_{\ell}$ is a harmonic Maa{\ss} form with non-cuspidal image under $\xi$, and we restrict ourselves to a more restricitve growth condition in Section \ref{Sec: formdef} to ensure convergence of our lift. It would be interesting to investigate different regularizations of our lift, and in particular lift the function $\Hs_{\ell}$.
\end{enumerate}
\end{remarks}

The paper is organized as follows. We establish the overall framework in Section \ref{Sec: prelims}. Section \ref{Sec: lift} is devoted to two evaluations of our theta lift, and to the proof of Theorem \ref{Thm: main}. In Section \ref{Sec: LB}, we compute the action of the Laplace-Beltrami operator on our theta lift and prove Theorem \ref{Thm: intro lhmf}. Lastly, Section \ref{Sec: CES} offers more details on the specialization to signature $(1,2)$, a proof of Theorem \ref{Cor: intro lhmf}, and an indication on Eichler--Selberg type relations for Cohen--Eisenstein series at the very end.

\section*{Acknowledgments}
The authors thank Markus Schwagenscheidt for fruitful discussions on the topics of the paper, and for many valuable comments on a previous version. The authors would also like to thank Ken Ono for suggesting the example of Cohen--Eisenstein series, and Kathrin Bringmann as well as the anonymous referee for helpful comments on a previous version of the paper.

\section{Preliminaries}\label{Sec: prelims}
We summarize some facts, which we require throughout.

\subsection{The Weil representation}
We recall the metaplectic double cover
\begin{align*}
\widetilde{\Gamma} \coloneqq \text{Mp}_2(\Z) \coloneqq \left\{ (\gamma, \phi) \colon \gamma = \left(\begin{smallmatrix} a & b \\ c & d \end{smallmatrix}\right)\in \slz, \ \phi\colon \H \rightarrow \C \text{ holomorphic}, \ \phi^2(\tau) = c\tau+d  \right\},
\end{align*}
of $\slz$, which is generated by the pairs
\begin{align*}
\widetilde{T} := \left(\left( \begin{matrix} 1 & 1 \\ 0 & 1 \end{matrix} \right),1\right), \qquad \widetilde{S} := \left(\left( \begin{matrix} 0 & -1 \\ 1 & 0 \end{matrix}\right) ,\sqrt{\tau}\right),
\end{align*}
where we fix a suitable branch of the complex square root throughout. Furthermore, we define $\widetilde{\Gamma}_\infty$ as the subgroup generated by $\widetilde{T}$. 

We let $L$ be an even lattice of signature $(r,s)$, and $Q$ be a quadratic form on $L$ with associated bilinear form $(\cdot,\cdot)_Q$. Moreover, we denote the dual lattice of $L$ by $L'$, and the group ring of $L' \slash L$ by $\C[L'\slash L]$. The group ring $\C[L'\slash L]$ has a standard basis, whose elements will be called $\e_{\mu}$ for $\mu \in L'\slash L$. We recall that there is a natural bilinear form  $\langle\cdot,\cdot\rangle$ on $\C[L'\slash L]$ defined by $\langle\e_{\mu},\e_{\nu}\rangle = \delta_{\mu,\nu}$. 

Equipped with this structure, the Weil representation $\rho_L$ of $\widetilde{\Gamma}$ associated to $L$ is defined on the generators by
\begin{align*}
\rho_L\left(\widetilde{T}\right)(\e_\mu) \coloneqq e(Q(\mu)) \e_\mu, \qquad
\rho_L\left(\widetilde{S}\right)(\e_\mu) \coloneqq \frac{e\left(\frac18(s-r)\right)}{\sqrt{|L'\slash L|}} \sum_{\nu \in L'\slash L} e(-(\nu,\mu)_Q) \e_{\nu},
\end{align*}
where we stipulate $e(x) \coloneqq e^{2\pi i x}$ throughout. We let $L^- \coloneqq (L,-Q)$, and call $\rho_{L^-}$ the dual Weil representation of $L$. 

\subsection{The generalized upper half plane and the invariant Laplacian}
We follow the introduction in \cite{bruinier2004borcherds}*{Sections 3.2, 4.1}, and let the signature of $L$ be $(2,s)$ here. We assume that $L$ is isotropic, i.e. that it contains a non-trivial vector $x$ of norm $0$, and by rescaling we may assume that it is primitive, that is if $x=cy$ for some $y \in L$ and $c \in \Z$ then $c = \pm 1$. Note that for $s\geq 3$ all lattices contain such an isotropic vector (see \cite{bor}*{Section 8}).

Let $z \in L$ be a primitive norm $0$ vector, and $z' \in L'$ with $(z,z')_Q = 1$. Let $K \coloneqq L \cap z^\perp \cap z'^\perp$. Let $d \in K$ be a primitive norm $0$ vector, and $d' \in K'$ with $(d,d')_Q = 1$. It follows that $D \coloneqq K \cap d^\perp \cap d'^\perp$ is a negative definite lattice, and we write
\begin{align*}
Z = (d'-Q(d')d)z_1+z_2d+z_3d_3+\ldots+z_{\ell}d_{\ell} \eqqcolon (z_1,z_2,\ldots,z_{\ell}) \in K \otimes \C,
\end{align*}
since $z_3d_3+\ldots+z_{\ell}d \in D \otimes \C$. Each $z_j$ has a real part $x_j$ and a imaginary part $y_j$, and we note that
\begin{align*}
Q(Y) \coloneqq Q(y_1,\ldots,y_{\ell}) = y_1y_2-y_3^2-y_4^2-\ldots-y_{\ell}^2.
\end{align*}
This gives rise to the generalized upper half plane 
\begin{align*}
\H_{\ell} \coloneqq \left\{Z \in K \otimes \C \colon y_1 > 0, Q(Y) > 0\right\} \cong \Gr(L).
\end{align*}
Letting
\begin{align*}
\partial_\mu \coloneqq \frac{\partial}{\partial z_\mu} = \frac{1}{2}\left(\frac{\partial}{\partial x_\mu} - i \frac{\partial}{\partial y_\mu}\right), \qquad \overline{\partial}_\mu \coloneqq \frac{\partial}{\partial \overline{z}_\mu} = \frac{1}{2}\left(\frac{\partial}{\partial x_\mu} + i \frac{\partial}{\partial y_\mu}\right),
\end{align*}
it can be shown that the invariant Laplacian on $\H_{\ell}$ has the coordinate representation \cite{Nak}
\begin{align*}
\Omega \coloneqq \sum_{\mu,\nu = 1}^{\ell} y_\mu y\nu \partial_\mu \overline{\partial}_\nu - Q(Y)\left(\partial_1 \overline{\partial}_2 + \overline{\partial}_1\partial_2 - \frac{1}{2}\sum_{\mu=3}^{\ell}\partial_\mu \overline{\partial}_\mu\right).
\end{align*}

\subsection{Maa{\ss} forms} \label{Sec: formdef}
Let $\kappa \in \frac{1}{2} \Z$, $(\gamma,\phi) \in \widetilde{\Gamma}$, and consider a function $f\colon \H \rightarrow \C[L' \slash L]$. The modular transformation in this setting is captured by the slash-operator
\begin{align*}
f\mid_{\kappa,\rho_{L}}(\gamma,\phi) (\tau) \coloneqq \phi(\tau)^{-2\kappa}\rho_{L}^{-1}(\gamma,\phi)f(\gamma\tau),
\end{align*}
which leads to vector-valued Maa{\ss} forms as follows \cite{bruinierfunke2004}.

\begin{definition}
Let $f\colon \H \rightarrow \C[L' \slash L]$ be smooth. Then $f$ is a Maa{\ss} form of weight $\kappa$ with respect to $\rho_L$ if it satisfies the following three conditions.
\begin{enumerate}[wide, labelwidth=0pt, labelindent=0pt]
\item We have $f\mid_{\kappa,\rho_{L}}(\gamma,\phi) (\tau) = f(\tau)$ for every $\tau \in \H$ and every $(\gamma,\phi) \in \widetilde{\Gamma}$.
\item The function $f$ is an eigenfunction of the weight $\kappa$ hyperbolic Laplace operator, which is explicitly given by
\begin{align*}
\Delta_{\kappa} \coloneqq -v^2\left(\frac{\partial^2}{\partial u^2} + \frac{\partial^2}{\partial v^2} \right) + i\kappa v\left(\frac{\partial}{\partial u} + i\frac{\partial}{\partial v} \right)
\end{align*}
\item There exists a polynomial\footnote{Such a polynomial is called the principal part of $f$.} in $q$ denoted by $P_f \colon \{0 < \vt{w} < 1 \} \to \C[L' \slash L]$ such that $f(\tau) - P_f(q) \in O\left(e^{-\varepsilon v}\right)$ as $v \to \infty$ for some $\varepsilon > 0$.
\end{enumerate}
We call $f$ a harmonic Maa{\ss} form if the eigenvalue equals $0$.
\end{definition}

We write $H_{\kappa,L}$ for the vector space of harmonic Maa{\ss} forms of weight $\kappa$ with respect to $\rho_L$, and $M_{\kappa,L}^! \subseteq H_{\kappa,L}$ for the subspace of weakly holomorphic vector valued modular forms. The subspace $S_{\kappa,L}^! \subseteq M_{\kappa,L}^!$ collects all forms that vanish at all cusps, and such forms are referred to as weakly holomorphic cusp forms.

Bruinier and Funke \cite{bruinierfunke2004} proved that a harmonic Maa{\ss} form $f$ of weight $\kappa \neq 1$ decomposes as a sum $f = f^++f^-$ of a holomorphic and a non-holomorphic part, whose Fourier expansions are of the shape
\begin{align*}
f^+(\tau) &= \sum_{\mu \in L' \slash L} \sum_{\substack{n \in \Q \\ n \gg - \infty}} c_f^+(\mu,n) q^n \e_{\mu}, \qquad f^-(\tau) = \sum_{\mu \in L' \slash L} \sum_{\substack{n \in \Q \\ n < 0}} c_f^-(\mu,n) \Gamma\left(1-\kappa,4\pi\vt{n}v\right)q^n \e_{\mu},
\end{align*}
where $\Gamma(t,x) \coloneqq \int_x^\infty u^{t-1} e^{-u} du$ denotes the incomplete Gamma function.

Harmonic Maa{\ss} forms can be inspected via the action of various differential operators. We require the antiholomorphic operator
\begin{align*}
\xi_\kappa \coloneqq 2iv^\kappa \overline{\frac{\partial}{\partial {\overline{\tau}}}},
\end{align*}
as well as the Maa{\ss} raising and lowering operators
\begin{align*}
R_{\kappa} \coloneqq 2i\frac{\partial}{\partial\tau} + \frac{\kappa}{v}, \qquad L_{\kappa} \coloneqq -2iv^{2} \frac{\partial}{\partial {\overline{\tau}}}.
\end{align*}
The operator $\xi_\kappa$ defines a surjective map from $H_{\kappa,L}$ to $S_{2-\kappa,L^-}^!$ \cite{bruinierfunke2004}. In particular, it intertwines with the slash operator introduced above, and the space $M_{\kappa,L}^!$ is precisely the kernel of $\xi_\kappa$ when restricted to $H_{\kappa,L}$. Hence, every $f \in H_{\kappa,L}$ has a cuspidal shadow in our case.

The operators $R_\kappa$ and $L_\kappa$ increase and decrease the weight $\kappa$ by $2$ respectively, but do not preserve the eigenvalue under $\Delta_{\kappa}$. For any $n \in \N_0$, we let
\begin{align*}
R_{\kappa}^0 &\coloneqq \mathrm{id}, \quad R_{\kappa}^n \coloneqq R_{\kappa+2n-2} \circ \ldots \circ R_{\kappa+2} \circ R_{\kappa}, \\
L_{\kappa}^0 &\coloneqq \mathrm{id}, \quad L_{\kappa}^n \coloneqq L_{\kappa-2n+2} \circ \ldots \circ L_{\kappa-2} \circ L_{\kappa}
\end{align*}
be the iterated Maa{\ss} raising and lowering operators, which increase or decrease the weight $\kappa$ by $2n$.

\begin{remark}
If one relaxes the growth condition (iii) to linear exponential growth, that is $f(\tau) \in O\left(e^{\varepsilon v}\right)$ as $v \to \infty$ for some $\varepsilon > 0$, then $f^-$ is permitted to have an additional (constant) term of the form $c_f^-(\mu,0) v^{1-\kappa} \e_\mu$. In this case, $\xi_\kappa$ maps such a form to a weakly holomorphic modular form instead of a weakly holomorphic cusp form.
\end{remark}

\subsubsection{Local Maa{\ss} forms}
Locally harmonic Maa{\ss} forms were introduced by Bringmann, Kane, Kohnen \cite{BKK} for negative weights, and independently by Hövel \cite{hoevel} for weight $0$. We generalize the exposition due to Bringmann, Kane, Kohnen here, and provide a definition in our setting on Grassmannians and for arbitrary eigenvalues.

\begin{definition}
A local Maa{\ss} form of weight $\kappa$ with closed exceptional set $X \subsetneq \H_{\ell}$ of measure zero is a function $f \colon \H_{\ell} \to \C[L' \slash L]$, which satisfies the following four properties:
\begin{enumerate}[label=(\roman*), wide, labelindent=0pt]
\item For all $(\gamma,\phi) \in \widetilde{\Gamma}$ and all $Z \in \H_{\ell}$ it holds that $f\mid_{\kappa,\rho_{L}}(\gamma,\phi) (Z) = f(Z)$.
\item For every $Z \in \H_{\ell} \setminus X$, there exists a neighborhood of $Z$, in which $f$ is real analytic and an eigenfunction of $\Omega$.
\item We have
\begin{align*}
f(Z) = \frac{1}{2} \lim_{\varepsilon \searrow 0} \left(f\left(Z+(i\varepsilon,0,\ldots,0)^t\right) + f\left(Z-(i\varepsilon,0,\ldots,0)^t\right)\right)
\end{align*}
for every $Z \in X$.
\item The function $f$ is of at most polynomial growth towards all cusps.
\end{enumerate}
\end{definition}
Paralleling the definition of harmonic Maa{\ss} forms, we call a local Maa{\ss} form locally harmonic if the eigenvalue from the second condition is $0$.

\subsection{Poincar\'{e} series}\label{Sec: MP}

\subsubsection{Weakly holomorphic Poincar{\'e} series}
Following Knopp, Mason \cite{KM}*{Section $3$}, we let $m \in \Z$, $\kappa \in \frac{1}{2}\N$ satisfying $\kappa > 2$, $\mu \in L'\slash L$, and define
\begin{align*}
\Fb_{\mu,m,\kappa}(\tau) \coloneqq \frac{1}{2}\sum_{(\gamma,\phi) \in \widetilde{\Gamma}_\infty \backslash \widetilde{\Gamma}} \left(e((m+1)\tau) \e_\mu\right) \mid_{\kappa, \rho_L} (\gamma,\phi).
\end{align*}
The authors of \cite{KM} prove that $\Fb_{\mu,m,\kappa}$ converges absolutely, and that it defines a weakly holomorphic modular form of weight $\kappa$ for $\rho_L$. In addition, they computed the Fourier expansion of $\Fb_{\mu,m,\kappa}$, which is of the shape
\begin{align*}
\Fb_{\mu,m,\kappa}(\tau) = \sum_{\nu \in L'\slash L} \left(\delta_{\mu,\nu}q^{m+1} + \sum_{n\geq 0} c(n)q^{n+1}\right)\e_{\nu}.
\end{align*}
The Fourier coefficients $c(n)$ can be found in \cite{KM}*{Theorem $3.2$} explicitly.

\subsubsection{Maa{\ss} Poincar\'{e} series}
We recall an important example of harmonic Maa{\ss} forms. To this end, let $\kappa \in -\frac{1}{2}\N$, let $M_{\mu, \nu}$ be the usual $M$-Whittaker function (see \cite{nist}*{§13.14}), and define the auxiliary function
\begin{align*}
\Mc_{\kappa,\sfr} (y) \coloneqq |y|^{-\frac{\kappa}{2}} M_{\sgn(y) \frac{\kappa}{2}, \sfr-\frac{1}{2}} (|y|), \qquad y \in \R\setminus\left\{0\right\}.
\end{align*}
We average $\Mc_{\kappa}$ over $\widetilde{\Gamma}$ as usual with respect to the parameters $\mu \in L' \slash L$, $m \in \N \setminus\left\{Q(\mu)\right\}$, and $\kappa$, $\sfr$. This yields the vector valued Maa{\ss} Poincar\'{e} series \cite{bruinier2004borcherds}
\begin{align*}
F_{\mu, m, \kappa,\sfr}(\tau) \coloneqq \frac{1}{2 \Gamma(2\sfr)} \sum_{(\gamma,\phi) \in \widetilde{\Gamma}_\infty \backslash \widetilde{\Gamma}} \left(\Mc_{\kappa,\sfr} (4 \pi m v) e(-mu) \e_\mu\right) \mid_{\kappa, \rho_L} (\gamma,\phi).
\end{align*}
By our choice of parameters and taking cosets, the series converges absolutely. The eigenvalue under $\Delta_{\kappa}$ is given by $(\sfr-\frac{\kappa}{2})(1-\sfr-\frac{\kappa}{2})$. Hence if $\sfr = \frac{\kappa}{2}$ or $\sfr = 1-\frac{\kappa}{2}$, then we have $F_{\mu, m, \kappa,\sfr} \in H_{\kappa,L}$. The principal part of $F_{\mu, m, \kappa,\sfr}$ is given by $e(-m\tau)(\e_{\mu}+\e_{-\mu})$ in this case, and $\xi_\kappa F_{\mu, -m, \kappa,\sfr}$ is a weight $2-\kappa$ cusp form.

Furthermore, the Maa{\ss} Poincar\'{e} series have the following useful property thanks to their simple principal part.
\begin{lemma} \label{lem:fdecomp}
Let $f \in H_{\kappa,L}$ with $\kappa \in -\frac{1}{2}\N$, and principal part
\begin{align*}
P_f(\tau) = \sum_{\mu \in L'\slash L} \sum_{n < 0} c_f^+(\mu,n) e(n\tau) \e_\mu \in \C[L'\slash L]\left[e(-\tau)\right].
\end{align*}
Then, we have
\begin{align*}
f(\tau) = \frac{1}{2} \sum_{\mu \in L' \slash L} \sum_{m > 0} c_f^+(\mu,-m) F_{\mu,m,\kappa,1-\frac{\kappa}{2}}(\tau).
\end{align*}
\end{lemma}

Additionally, we require the following computational lemma, which is taken from \cite{ANBMS}*{Lemma 2.1}, and follows inductively from \cite{bruinier2020greens}*{Proposition 3.4}.
\begin{lemma} \label{lem:mpraise}
For any $n \in \N_0$ it holds that
\begin{align*}
R_\kappa^n \left(F_{\mu, m, \kappa,\sfr}\right)(\tau) = (4\pi m)^n \frac{\Gamma\left(\sfr+n+\frac{\kappa}{2}\right)}{\Gamma\left(\sfr+\frac{\kappa}{2}\right)} \ F_{\mu, m,\kappa+2n,\sfr}(\tau).
\end{align*}
\end{lemma}

\subsection{Restriction, trace maps, and Rankin--Cohen brackets}
As before, we fix an even lattice $L$. We let $A_{\kappa,L}$ be the space of smooth functions $f \colon \H \to \C[L' \slash L]$, which are invariant under the weight $\kappa$ slash operator with respect to the representation $\rho_L$. Moreover, let $K \subseteq L$ be a finite index sublattice. Hence, we have $L' \subseteq K'$, and thus 
\begin{align*}
L \slash K \subseteq L'\slash K \subseteq K'\slash K.
\end{align*}
This induces a map
\begin{align*}
L' \slash K &\to L' \slash L \\
\mu &\mapsto \bar{\mu}.
\end{align*}
If $\mu \in K'\slash K$, $f \in A_{\kappa, L}$, $g \in A_{\kappa,K}$, and $\mu$ is a fixed preimage of $\bar{\mu}$ in $L'\slash K$, we define
\begin{align*}
(f_K)_\mu := \begin{cases}
f_{\bar{\mu}} & \text{ if } \mu \in L' \slash K, \\
0 & \text{ if } \mu \not\in L' \slash K,
\end{cases} 
\qquad 
\left(g^L\right)_{\bar{\mu}} = \sum_{\alpha \in L \slash K} g_{\alpha + \mu},
\end{align*}

The following lemma may be found in \cite{bruinierFaltings}*{Section 3}.
\begin{lemma} \label{lem:restrace}
In the notation above, there are two natural maps
\begin{align*}
\begin{array}{rlrl}
\textup{res}_{L\slash K} \colon A_{\kappa, L} &\to A_{\kappa, K}, & \tr_{L\slash K} \colon A_{\kappa, K} &\to A_{\kappa, L}, \\
f &\mapsto f_K & g &\mapsto g^L
\end{array}
\end{align*}
satisfying
\begin{align*}
\langle f, \overline{g}^L \rangle = \langle f_K, \overline{g} \rangle
\end{align*}
for any $f \in A_{\kappa,L}$, $g \in A_{\kappa,K}$.
\end{lemma}

Let $\kappa$, $\ell \in \frac{1}{2} \Z$, $f \in A_{\kappa,K}$, $g \in A_{\ell,L}$. Writing
\begin{align*}
f = \sum_{\mu}f_\mu \e_\mu, \qquad g = \sum_{\nu}g_\nu \e_\nu,
\end{align*}
and letting $n \in \N_0$, we define the tensor product of $f$ and $g$ as well as the $n$-th Rankin--Cohen bracket of $f$ and $g$ as
\begin{align*}
f \otimes g &\coloneqq \sum_{\mu , \nu} f_\mu g_\nu \e_{\mu + \nu} \in A_{\kappa+\ell,K \oplus L}, \\
[f,g]_n &\coloneqq \frac{1}{(2\pi i)^n}\sum_{\substack{r,s \geq 0\\ r+s=n}} \frac{ (-1)^r \Gamma(\kappa+n) \Gamma(\ell+n)}{\Gamma(s+1) \Gamma(\kappa+n-s) \Gamma(r+1) \Gamma(\ell+n-r)}  f^{(r)} \otimes g^{(s)},
\end{align*}
where $f^{(r)}$ and $g^{(s)}$ are usual higher derivatives of $f$ and $g$. Then we have the following vector-valued analogue of \cite{bruinier2020greens}*{Proposition 3.6}.
\begin{lemma}\label{Lem: RC Brackets}
Let $f \in H_{\kappa,L_1}$ and $g \in H_{\ell,L_2}$. For $n \in \N_0$ it holds that
\begin{align*}
(-4\pi)^n L_{\kappa+\ell+2n}\left([f,g]_n \right)= \frac{\Gamma(\kappa+n) }{n! \ \Gamma(\kappa)} L_\kappa (f) \otimes R_\ell^n (g) + (-1)^n \frac{ \Gamma(\ell+n)}{n! \ \Gamma(\ell)} R_\kappa^n (f) \otimes L_\ell (g).
\end{align*}
\end{lemma}

Finally, we have the following lemma, which can be verified straightforwardly (see \cite{ANBMS}*{Proof of Theorem 4.1}).
\begin{lemma}\label{Lem: tech}
	Let $h$ be a smooth function, $g$ be holomorphic, and $\kappa$, $\ell \in \R$. Then it holds that
	\begin{align*}
	R_{\ell-\kappa}(v^{\kappa}\overline{g} \otimes h) = v^k\overline{g} \otimes R_{\ell} h.
	\end{align*}
\end{lemma}

\subsection{Theta functions and special points}\label{Sec: theta functions}

We fix an even lattice $L$ of signature $(r,s)$, and extend the quadratic form on $L$ to $L \otimes \R$ in the natural way. We denote the orthogonal projection of $\lambda \in L+\mu$ onto the linear subspaces spanned by $z$ and its orthogonal complement with respect to $(\cdot,\cdot)_Q$ by $\lambda_z$ and $\lambda_{z^\perp}$ respectively. In other words, we have 
\begin{align*}
L \otimes \R = z \oplus z^\perp, \qquad \lambda = \lambda_z + \lambda_{z^\perp}.
\end{align*}
Let $\Gr(L)$ be the Grassmannian of $r$-dimensional subspaces of $L \otimes \R$. Let $Z \subseteq \Gr(L)$ be the set of all such subspaces on which $Q$ is positive definite. One can endow $Z$ with the structure of a smooth manifold.

Let $p_r \colon \R^{r,0} \to \C$, and $p_s \colon \R^{0,s} \to \C$ be spherical polynomials, which are homogeneous of degree $d^+$, $d^- \in \N_0$ respectively. Define
\begin{align*}
p \coloneqq p_r\otimes p_s,
\end{align*}
and let $\psi \colon L \otimes \R \to \R^{r,s}$ be an isometry. We set
\begin{align*}
z \coloneqq \psi^{-1}(\R^{r,0}) \in Z, \qquad z^{\perp} = \psi^{-1}(\R^{0,s}).
\end{align*}

For a positive-definite lattice $(K,Q)$  of rank $n$, and a homogeneous spherical polynomial $p$ of degree $d$, we define the usual theta function
\begin{align*}
	\Theta_K(\tau,\psi_K,p) \coloneqq \sum_{\lambda \in K'} p(\psi_K(\lambda)) e\left({Q(\lambda) \tau}\right),
\end{align*}
where $\psi_K$ is the isometry associated to $K$. It is a holomorphic modular form of weight $\frac{n}{2}+d$ for $\rho_K$. If the isometry is trivial, we write $\Theta_K(\tau,p)$.

Following Borcherds \cite{bor} and H\"{o}vel \cite{hoevel}, we define the general Siegel theta function as follows.\footnote{In fact, Borcherds considered a slightly more general theta function, where the polynomial $p$ does not necessarily vanish under $\Delta_\kappa$. For us however, this more general case would not yield spherical theta functions as we desire.}
\begin{definition}
Let $\tau \in \H$ and assume the notation above. Then we put
\begin{align*}
\Theta_L(\tau,\psi,p) \coloneqq v^{\frac{s}{2}+d^-}\sum_{\mu \in L'\slash L} \sum_{\lambda \in L+\mu} p\left(\psi(\lambda)\right) e\left(Q(\lambda_{z})\tau + Q(\lambda_{z^\perp})\overline{\tau}\right) \mathfrak{e}_\mu.
\end{align*}
\end{definition}

One can check that the function $\Theta_L$ converges absolutely on $\H \times Z$. The following result is \cite{hoevel}*{Satz 1.55}, which follows directly from \cite{bor}*{Theorem 4.1}.
\begin{lemma}
Let $(\gamma,\phi) \in \widetilde{\Gamma}$. Then we have
\begin{align*}
\Theta_L(\gamma\tau, \psi, p) = \phi(\tau)^{r+2d^+-(s+2d^-)} \rho_{L}(\gamma,\phi) \Theta_L(\tau,\psi,p).
\end{align*}
\end{lemma}

Thus, we define
\begin{align*}
k \coloneqq \frac{r-s}{2} + d^+-d^-.
\end{align*}
The following terminology is borrowed from \cite{BS}.
\begin{definition}
An element $w \in \Gr(L)$ is called a special point if it is defined over $\Q$ that is $w \in L \otimes \Q$.
\end{definition}

We observe that if $w$ is a special point, then $w^\perp$ is a special point as well. This yields the splitting
\begin{align*}
L \otimes \Q = w \oplus w^\perp,
\end{align*}
which in turn yields the positive and negative definite lattices
\begin{align*}
P \coloneqq L \cap w, \qquad N \coloneqq L \cap w^\perp.
\end{align*}
Clearly, $P \oplus N$ is a sublattice of $L$ of finite index, and according to Lemma \ref{lem:restrace}, the theta functions associated to both lattices are related by
\begin{align*}
\Theta_L = (\Theta_{P \oplus N})^L.
\end{align*}
We identify $\C[(P\oplus N)'\slash (P\oplus N)]$ with $\C[P'\slash P] \otimes \C[N' \slash N]$, and let $\psi_P$, $\psi_N$ be the restrictions of $\psi$ onto $P$, $N$ respectively. Consequently, we have the splitting
\begin{align*}
\Theta_{P\oplus N}(\tau,\psi,p) = \Theta_P(\tau,\psi_P,p_r) \otimes v^{\frac{s}{2}+d^-}\overline{\Theta_{N^-}(\tau,\psi_N,p_s)}
\end{align*}
at a special point $w$, which can be verified straightforwardly. Furthermore, we observe that $\Theta_P(\tau,\psi_P,p_r)$ is holomorphic and of weight $\frac{r}{2}+d^+$ as a function of $\tau$, while $v^{\frac{s}{2}+d^-}\overline{\Theta_{N^-}(\tau,\psi_N,p_s)}$ is of weight $-\frac{s}{2}-d^-$ with respect to $\tau$.

\subsection{Serre duality}
The following result can be found in \cite[Proposition 2.5]{LiSch} for instance.

\begin{proposition}[Serre duality]\label{Prop: Serre}
Let $L$ be an even lattice, and $\kappa \in \frac{1}{2}\Z$. Assume that
\begin{align*}
g(\tau) = \sum_{h \in L'\slash L}\sum_{n \geq 0}c_g(h,n)e(n\tau)\e_h
\end{align*}
is bounded at the cusp $i\infty$. Then $g$ is a holomorphic modular form of weight $\kappa$ for the Weil representation $\rho_L$ if and only if we have
\begin{align*}
\sum_{h \in L'\slash L}\sum_{n \geq 0}c_g(h,n)c_f(h,-n) = 0
\end{align*}
for every weakly holomorphic modular form $f$ of weight $2-\kappa$ for $\overline{\rho}_L$.
\end{proposition}

\section{The theta lift}\label{Sec: lift}

We consider the theta lift $\Psi^{\reg}_j(f,z)$ and evaluate it in two different ways. Using Serre duality goes back to Borcherds \cite{bor2}.

\subsection{Evaluation in terms of ${}_2F_1$}

We begin by evaluating the higher modified lift as a series involving Gauss hypergeometric functions as follows.

\subsubsection{Evaluating the theta lift of Maa{\ss}--Poincar\'{e} series for general spectral parameters}
Let $\sfr \in \C$ be such that
\begin{align*}
F_{m, \kappa,\sfr}(\tau) \coloneqq \sum_{\mu \in L' \slash L} F_{\mu, m, \kappa,\sfr}(\tau)
\end{align*}
converges absolutely, that is $\re(\sfr) > 1-\frac{\kappa}{2}$.

\begin{theorem} \label{Theorem: theta lift of MPseries}
We have
\begin{multline*}
\Psi^{\reg}_j \left(F_{m, k-2j,\sfr}, z\right) = (4\pi m)^{j+1-k-\frac{s}{2}-d^-}\frac{\Gamma\left(\sfr+\frac{k}{2}\right)\Gamma\left(\frac{k+s}{2}+d^--1+\sfr\right)}{2 \Gamma(2-k+2j)\Gamma\left(\sfr+\frac{k}{2}-j\right)} \\
\times \sum_{\mu \in L'\slash L} \sum_{\substack{\lambda \in L+\mu \\ Q(\lambda)=-m}} \overline{p(\psi(\lambda))} \left(\frac{Q(\lambda)}{Q\left(\lambda_{z^{\perp}}\right)}\right)^{\frac{k+s}{2}+d^--1+\sfr} {}_2F_1\left(k+\sfr, \frac{k+s}{2}+d^--1+\sfr; 2\sfr; \frac{Q(\lambda)}{Q\left(\lambda_{z^{\perp}}\right)}\right).
\end{multline*}
\end{theorem}

\begin{remark}
Choosing the homogeneous polynomial in the theta kernel function to be the constant function $1$ and computing the action of $R_{k-2j}^j$ on $F_{m,k-2j,\sfr}$ by Lemma \ref{lem:mpraise}, this result becomes \cite[Theorem 2.14]{bruinier2004borcherds}.
\end{remark}

\begin{proof}
We summarize the argument from \cite[Theorem 2.14]{bruinier2004borcherds} for convenience of the reader. We need to evaluate
\begin{align*}
\Psi^{\reg}_j \left(F_{m, k-2j,\sfr}, z\right) = \int_{\Fc}^{\reg} \left\langle R_{k-2j}^{j}(F_{m, k-2j,\sfr})(\tau) , \overline{\Theta_L(\tau,\psi,p)} \right\rangle v^{k} d\mu(\tau)
\end{align*}
Consequently, we compute the action of the raising operator first, and have
\begin{align*}
\Psi^{\reg}_j \left(F_{m, k-2j,\sfr}, z\right) = (4\pi m)^j \frac{\Gamma\left(\sfr+\frac{k}{2}\right)}{\Gamma\left(\sfr+\frac{k}{2}-j\right)} \int_{\Fc}^{\reg} \left\langle (F_{m, k,\sfr})(\tau) , \overline{\Theta_L(\tau,\psi,p)} \right\rangle v^{k} d\mu(\tau)
\end{align*}
by Lemma \ref{lem:mpraise}. Secondly, we insert the definitions of both functions, and unfold the integral, obtaining
\begin{multline*}
\Psi^{\reg}_j \left(F_{m, k-2j,\sfr}, z\right) = \frac{(4\pi m)^j\Gamma\left(\sfr+\frac{k}{2}\right)}{2 \Gamma(2-k+2j)\Gamma\left(\sfr+\frac{k}{2}-j\right)} \sum_{\mu \in L'\slash L} \sum_{\lambda \in L+\mu} \overline{p(\psi(\lambda))} \\
\times \int_0^1 \int_0^{\infty} (4\pi m v)^{-\frac{k}{2}} M_{-\frac{k}{2},\sfr-\frac{1}{2}}(4\pi m v) e(-mu) \overline{e\left( Q(\lambda_z)\tau + Q(\lambda_{z^\perp})\overline{\tau} \right)} v^{\frac{s}{2}+d^-+k-2} dv du.
\end{multline*}
Third, we compute the integral over $u$ using that $\overline{e(w)} = e(-\overline{w})$, and that
\begin{align*}
\int_0^1 e(-mu) e\left(-Q(\lambda_z)u - Q(\lambda_{z^\perp})u \right) du = \begin{cases}
1 & \text{ if } Q(\lambda_z) + Q(\lambda_{z^{\perp}}) =-m, \\
0 & \text{ else}.
\end{cases}
\end{align*}
Hence, we obtain
\begin{multline*}
\Psi^{\reg}_j \left(F_{m, k-2j,\sfr}, z\right) = \frac{(4\pi m)^{j-\frac{k}{2}}\Gamma\left(\sfr+\frac{k}{2}\right)}{2 \Gamma(2-k+2j)\Gamma\left(\sfr+\frac{k}{2}-j\right)} \sum_{\mu \in L'\slash L} \sum_{\substack{\lambda \in L+\mu \\ Q(\lambda)=-m}} \overline{p(\psi(\lambda))} \\ \times \int_0^{\infty} M_{-\frac{k}{2},\sfr-\frac{1}{2}}(4\pi m v) e^{-2\pi v \left(Q(\lambda_z) - Q(\lambda_{z^\perp})\right)} v^{\frac{s+k}{2}+d^--2} dv.
\end{multline*}
The integral is a Laplace transform. Using that $\frac{m}{2m}+\frac{Q(\lambda_z) - Q(\lambda_{z^\perp})}{2m} =  \frac{Q\left(\lambda_{z^{\perp}}\right)}{Q(\lambda)}$ along with \cite{nist}*{item 13.23.1}, it evaluates
\begin{align*}
&\int_0^{\infty} M_{-\frac{k}{2},\sfr-\frac{1}{2}}(4\pi m v) e^{-2\pi v \left(Q(\lambda_z) - Q(\lambda_{z^\perp}) \right)} v^{\frac{k+s}{2}+d^--2} dv \\
&= \frac{(4\pi m)^{1-\frac{k+s}{2}-d^-} \Gamma\left(\frac{k+s}{2}+d^--1+\sfr\right)}{\left(\frac{Q(\lambda_z) - Q(\lambda_{z^\perp})}{2m} + \frac{1}{2}\right)^{\frac{k+s}{2}+d^--1+\sfr}} {}_2F_1\left(k+\sfr, \frac{k+s}{2}+d^--1+\sfr; 2\sfr; \frac{1}{\frac{1}{2}+\frac{Q(\lambda_z) - Q(\lambda_{z^\perp})}{2m}}\right).
\end{align*}
We recall $Q(\lambda) = Q(\lambda_z)+Q(\lambda_{z^\perp}) = -m$, and rewrite the argument of the hypergeometric function to
\begin{align*}
\frac{m}{2m}+\frac{Q(\lambda_z) - Q(\lambda_{z^\perp})}{2m} = \frac{Q\left(\lambda_{z^{\perp}}\right)}{Q(\lambda)}.
\end{align*}
Thus, we arrive at
\begin{multline*}
\Psi^{\reg}_j \left(F_{m, k-2j,\sfr}, z\right) = (4\pi m)^{j+1-k-\frac{s}{2}-d^-}\frac{\Gamma\left(\sfr+\frac{k}{2}\right)\Gamma\left(\frac{k+s}{2}+d^--1+\sfr\right)}{2 \Gamma(2-k+2j)\Gamma\left(\sfr+\frac{k}{2}-j\right)} \\
\times \sum_{\mu \in L'\slash L} \sum_{\substack{\lambda \in L+\mu \\ Q(\lambda)=-m}} \overline{p(\psi(\lambda))} \left(\frac{Q(\lambda)}{Q\left(\lambda_{z^{\perp}}\right)}\right)^{\frac{k+s}{2}+d^--1+\sfr} {}_2F_1\left(k+\sfr, \frac{k+s}{2}+d^--1+\sfr; 2\sfr; \frac{Q(\lambda)}{Q\left(\lambda_{z^{\perp}}\right)}\right),
\end{multline*}
as claimed.
\end{proof}

Combining the previous result with Lemma \ref{lem:fdecomp} yields the following consequence.
\begin{corollary}\label{cor: theta lift as 2F1}
Let $j \in \N_0$, and $f \in H_{k-2j,L}$. Assume that $k-2j < 0$. Then we have
\begin{multline*}
\Psi^{\reg}_j \left(f, z\right) = \frac{(4\pi)^{j+1-k-\frac{s}{2}-d^-} j! \ \Gamma\left(\frac{s}{2}+d^-+j\right)}{4\Gamma(2-k+2j)} \sum_{\substack{\lambda \in L' \\ Q(\lambda) <0}} c_f^+(\lambda,Q(\lambda))  \overline{p(\psi(\lambda))} \\
\times \frac{\vt{Q(\lambda)}^{2j+1-k}}{\vt{Q(\lambda_{z^\perp})}^{\frac{s}{2}+j+d^-}} {}_2F_1\left(1+j, \frac{s}{2}+d^-+j;2-k+2j;\frac{Q(\lambda)}{Q(\lambda_{z^\perp})}\right).
\end{multline*}
\end{corollary}

\begin{proof}
Since the weight of $f$ is negative, we have
\begin{align*}
f(\tau)= \frac{1}{2} \sum_{h \in L'\slash L} \sum_{m \geq 0} c_f^+(h,-m) F_{h,m,k-2j,1-\frac{k}{2}+j}(\tau)
\end{align*}
according to Lemma \ref{lem:fdecomp}, and we observe that the term corresponding to $m=0$ will vanish due to $c_f^+(h,0) = 0$ by our more restrictive growth condition on Maa{\ss} forms. Consequently, we have
\begin{align*}
\Psi^{\reg}_j \left(f, z\right) = \frac{1}{2} \sum_{\mu \in L'\slash L} \sum_{m > 0} c_f^+(\mu,-m) \Psi^{\reg}_j \left(F_{\mu,m,k-2j,1-\frac{k}{2}+j}, z\right).
\end{align*}
We insert the spectral parameter $\sfr = 1-\frac{k-2j}{2}$ into Theorem \ref{Theorem: theta lift of MPseries}, which yields the claim.
\end{proof}

\subsection{Evaluation in terms of the constant term in a Fourier expansion}
Next we determine the lift as a constant term in a Fourier expansion plus a certain boundary integral that vanishes for a certain class of input function.
\begin{theorem}\label{Theorem: theta lift at CM points}
Let $f \in H_{k-2j,L}$ and $w$ be a special point, and $\Gc_P^+$ be the holomorphic part of a preimage of $\Theta_P$ under $\xi_{2-(\frac{r}{2}+d^+)}$. Then we have
	\begin{multline*}	
\Psi^{\reg}_j \left(f, w\right) =  \frac{j! (4\pi)^j \Gamma\left(2-\frac{r}{2}-d^+\right)}{\Gamma\left(2-\frac{r}{2}-d^++j\right)} \left( \CT\left(\left\langle f_{P\oplus N}(\tau),  \left[\Gc_P^+(\tau),\Theta_{N^-}(\tau)\right]_j  \right\rangle \right) \right. \\ \left. - \int_{\Fc}^{\reg} \left\langle L_{k-2j} \left(f_{P\oplus N}\right)(\tau), \left[\Gc_P^+(\tau),\Theta_{N^-}(\tau)\right]_j  \right\rangle v^{-2} d\tau\right).
	\end{multline*}
	\end{theorem}

\begin{remark}
In general, the coefficients of $\Gc_P^+$ are expected to be transcendental. However, in weight $\frac{1}{2}$ and $\frac{3}{2}$ the function $\Gc_P^+$ may be chosen to have rational coefficients - a situation which is expected to also hold for $\xi$-preimages of CM modular forms. It is therefore expected that one obtains rationality (up to powers of $\pi$) of the modified higher lift only in these cases, and stipulating that $f$ is weakly holomorphic meaning that the final integral vanishes.
\end{remark}

By a slight abuse of notation, we write $\Theta_L(\tau,w,p)$ for the theta function evaluated at an isometry $\psi$ that produces a special point $w$.

\begin{proof}[Proof of Theorem \ref{Theorem: theta lift at CM points}]
We restrict to special points $w \in \Gr(L)$. This enables us to write
\begin{align*}
\left\langle R_{k-2j}^{j}(f)(\tau) , \overline{\Theta_L(\tau,w,p)} \right\rangle = \left\langle R_{k-2j}^{j}(f_{P\oplus N})(\tau) , \overline{\Theta_{P\oplus N}(\tau,w,p)} \right\rangle.
\end{align*}
Next, we use that the raising and lowering operator are adjoint to each other (cf. \cite{bruinier2004borcherds}*{Lemma 4.2}), which gives
\begin{align*}
\Psi^{\reg}_j \left(f, w\right) = \int_{\Fc}^{\reg} \left\langle f_{P\oplus N}(\tau) , L_{k}^{j-1}\left(\overline{\Theta_{P\oplus N}(\tau,w,p)}\right)\right\rangle v^{k-2} d\tau.
\end{align*}
We observe that the boundary terms disappear in the same fashion as during the proof of \cite{bruinier2004borcherds}*{Lemma 4.4}. Next, we rewrite
\begin{align*}
\Psi^{\reg}_j \left(f, w\right) = (-1)^j\int_{\Fc}^{\reg} \left\langle f_{P\oplus N}(\tau), R_{-k}^{j}\left(\overline{\Theta_{P\oplus N}(\tau,w,p)}v^{k}\right) \right\rangle v^{-2} d\tau,
\end{align*}
and recall that
\begin{align*}
\Theta_{P\oplus N}(\tau,w,p) = \Theta_P(\tau,p_r) \otimes v^{\frac{s}{2}+d^-}\overline{\Theta_{N^-}(\tau,p_s)} = v^{\frac{s}{2}+d^-}\Theta_P(\tau,p_r) \otimes \overline{\Theta_{N^-}(\tau,p_s)}.
\end{align*}
Consequently, we obtain
\begin{align*}
R_{-k}^{j}\left(\overline{\Theta_{P\oplus N}(\tau,w,p)}v^{k}\right) &= R_{-k}^j \left(v^{k+\frac{s}{2}+d^-}\overline{\Theta_P(\tau,p_r)}\otimes \Theta_{N^-}(\tau,p_s)\right) \\
&= v^{k+\frac{s}{2}+d^-}\overline{\Theta_P(\tau,p_r)} \otimes \left(R_{\frac{s}{2}+d^-}^{j}\Theta_{N^-}(\tau,p_s)\right),
\end{align*}
by Lemma \ref{Lem: tech}. In particular we note that $v^{k+\frac{s}{2}+d^-}\overline{\Theta_P(\tau,p_r)}$ has weight $-k-\frac{s}{2}-d^- = -\frac{r}{2}-d^+$.

We choose a preimage $\Gc_P$ of $\Theta_P(\tau,p_r)$ under $\xi_{2-(\frac{r}{2}+d^+)}$, namely
\begin{align*}
\Theta_P(\tau,p_r) = \xi_{2-\frac{r}{2}-d^+} \Gc_P(\tau) = v^{-\frac{r}{2}-d^+} \overline{L_{2-\frac{r}{2}-d^+}} \Gc_P,
\end{align*}
which yields
\begin{align*}
R_{-k}^{j}\left(\overline{\Theta_{P\oplus N}(\tau,w,p)}v^{k}\right) = L_{2-\frac{r}{2}-d^+} \Gc_P(\tau) \otimes \left(R_{\frac{s}{2}+d^-}^{j}\Theta_{N^-}(\tau,p_s)\right)
\end{align*}
We apply the computation of the Rankin--Cohen brackets given in Lemma \ref{Lem: RC Brackets} noting that $L_{\ell}\Theta_{N^-} = 0$, and that it suffices to deal with the holomorphic part $\Gc_P^+$ of $\Gc_P$ (both by virtue of holomorphicity in computing the Rankin--Cohen bracket). Thus,
\begin{align*}
R_{-k}^{j}\left(\overline{\Theta_{P\oplus N}(\tau,w,p)}v^{k}\right) = \frac{j! (-4\pi)^j \Gamma(2-k)}{\Gamma\left(2-k+j\right)} v^{-\frac{s}{2}-d^-} L_{2-k+\frac{s}{2}+d^-+2j} \left[\Gc_P^+(\tau),\Theta_{N^-}(\tau,p_s)\right]_j.
\end{align*}
Hence, the theta lift becomes
\begin{align*}
\Psi^{\reg}_j \left(f, w\right) = \frac{j! (4\pi)^j \Gamma\left(2-\frac{r}{2}-d^+\right)}{\Gamma\left(2-\frac{r}{2}-d^++j\right)} \int_{\Fc}^{\reg} \left\langle f_{P\oplus N}(\tau) , L_{2-k+2j} \left[\Gc_P^+(\tau),\Theta_{N^-}(\tau,p_s)\right]_j\right\rangle v^{-2} d\tau.
\end{align*}
The last step is to apply Stokes' theorem, compare the proof of \cite{bruinier2004borcherds}*{Lemma 4.2} for example, which yields
\begin{multline*}
\Psi^{\reg}_j \left(f, w\right) = \frac{j! (4\pi)^j \Gamma\left(2-\frac{r}{2}-d^+\right)}{\Gamma\left(2-\frac{r}{2}-d^++j\right)} \left(\lim_{T \rightarrow \infty}  \int_{iT}^{1+iT} \left\langle f_{P\oplus N}(\tau),  \left[\Gc_P^+(\tau),\Theta_{N^-}(\tau,p_s)\right]_j\right\rangle v^{-2} d\tau \right. \\ \left. - \int_{\Fc}^{\reg} \left\langle L_{k-2j} \left(f_{P\oplus N}\right)(\tau), \left[\Gc_P^+(\tau),\Theta_{N^-}(\tau,p_s)\right]_j  \right\rangle v^{-2} d\tau\right),
\end{multline*}
utilizing again that boundary terms vanish. We observe that the left integral can be regarded as the Fourier coefficient of index $0$ in the Fourier expansion of the integrand, see the bottom of page $14$ in \cite{BS}. This proves the claim.
\end{proof}

We end this section by noting that to obtain recurrence relations, as in \cite{BS}, one would need to compute the Fourier expansion of the lift. In general, this is a lengthy but straightforward process, and since we do not require it in this paper we omit the details. In essence, one follows the calculations of Borcherds \cite{bor} by using Lemma \ref{lem:mpraise}. A resulting technicality is to then take care of the different spectral parameter. One may overcome this by relating the coefficients of Maa{\ss}-Poincar\'{e} series to those with other spectral parameters, again using the action of the iterated Maa{\ss} raising operator as in Lemma \ref{lem:mpraise}.

\subsection{Eichler--Selberg type relations} \label{sec:ESrelations}
We now prove a refined version of Theorem \ref{Thm: main}. To this end, we define the function
\begin{multline*}
\Lambda_L(\psi,p,j) \coloneqq  \frac{(4\pi)^{1-\frac{r}{2}-d^+} \Gamma\left(\frac{s}{2}+j+d^-\right)\Gamma\left(2-\frac{r}{2}-d^++j\right)}{4\Gamma(2-k+2j)\Gamma\left(2-\frac{r}{2}-d^+\right)} 
\sum_{\substack{m \geq 1 \\ \lambda \in L' \\ Q(\lambda) =-m}} \overline{p(\psi(\lambda))} \frac{\vt{Q(\lambda)}^{2j+1-k}}{\vt{Q(\lambda_{z^\perp})}^{\frac{s}{2}+j+d^-}} \\ \times{}_2F_1\left(1+j, \frac{s}{2}+j+d^-;2-k+2j;\frac{Q(\lambda)}{Q(\lambda_{z^\perp})}\right) q^m
\end{multline*}
for $j>0$. We write 
\begin{align*}
\Gc_P^+(\tau) = \sum_{\mu \in L'\slash L} \sum_{n \gg -\infty} a(n)q^n \e_{\mu},
\end{align*}
and furthermore define
\begin{align*}
\mathscr{G}_P^+(\tau) \coloneqq \Gc_P^+(\tau) - \sum_{\mu \in L'\slash L} \sum_{n < 0} a(n) \Fb_{\mu,n-1,2j+2-k}(\tau).
\end{align*}

Since one may add any weakly holomorphic modular form of appropriate weight for $\rho_L$ to $\Gc_P^+$, Theorem \ref{Thm: main} follows directly from the following result (noting that the linear combination of Maa{\ss} Poincar\'{e} series may change).
\begin{theorem}\label{Thm: ES}
Let $L$ be an even lattice of signature $(r,s)$, let $p$ be as before, and $w$ be a special point defined by the isometry $\psi$. Let $j > 0$ and $k$ be such that $2j+2-k > 2$. Then the function
\begin{align*}
\left[\mathscr{G}_P^+(\tau),\Theta_{N^-}(\tau,p_s)\right]_j^L - \Lambda_L(\psi,p,j)
\end{align*}
is a holomorphic vector-valued modular form of weight $2j+2-k$ for $\rho_L$.
\end{theorem}

\begin{remarks}
\
\begin{enumerate}
\item This provides the general vector-valued analogue, assuming that the lattice is chosen such that $2j+2-k > 2$, of Mertens' scalar-valued results in weight $\frac{1}{2}$ and $\frac{3}{2}$ \cite{Mer1}. 
\item Note that the slight correction of $\Gc_P^+$ by Poincar{\'e} series was missing in \cite{Ma}. 
\item In certain cases the hypergeometric function may be simplified (for example, the $n=1$ case as in \cites{BS,Ma}, which yields a form very similar to Mertens' scalar-valued result). It appears to be possible that one should be able to prove the same results via holomorphic projection acting on vector-valued modular forms (see \cite{IRR}) in much the same way as Mertens' original scalar valued proofs in \cite{Mer1}.
\end{enumerate}
\end{remarks}

\begin{proof}[Proof of Theorem \ref{Thm: ES}]
Let $f$ be a weakly holomorphic form of weight $k-2j$ with Fourier coefficients $c_f^+$. By construction, the form $\mathscr{G}_P^+$ is holomorphic at $i\infty$, and hence 
\begin{align*}
\CT\left(\left\langle f_{P\oplus N}(\tau),  \left[\mathscr{G}_P^+(\tau),\Theta_{N^-}(\tau,p_s)\right]_j^L  \right\rangle \right)
\end{align*}
contains only the Fourier coefficients of non-positive index of $f$. We note that $L_{k-2j} f = 0$, and subtract the resulting expressions of the lift from Corollary \ref{cor: theta lift as 2F1} and Theorem \ref{Theorem: theta lift at CM points}. We obtain
\begin{multline*}
0 = \CT\left(\left\langle f_{P\oplus N}(\tau),  \left[\mathscr{G}_P^+(\tau),\Theta_{N^-}(\tau,p_s)\right]_j^L  \right\rangle \right) - \frac{(4\pi)^{1-\frac{r}{2}-d^+} \Gamma\left(\frac{s}{2}+j+d^-\right)\Gamma\left(2-\frac{r}{2}-d^++j\right)}{4\Gamma(2-k+2j)\Gamma\left(2-\frac{r}{2}-d^+\right)} \\ 
\times \sum_{\substack{m \geq 1 \\ \lambda \in L' \\ Q(\lambda) =-m}} c_f^+(\lambda,-m) \overline{p(\psi(\lambda))} \frac{\vt{Q(\lambda)}^{2j+1-k}}{\vt{Q(\lambda_{z^\perp})}^{\frac{s}{2}+j+d^-}} {}_2F_1\left(1+j, \frac{s}{2}+j+d^-;2-k+2j;\frac{Q(\lambda)}{Q(\lambda_{z^\perp})}\right).
\end{multline*}
The Rankin--Cohen bracket is bilinear, and a linear combination of vector-valued Poincar{\'e} series is modular itself. We apply Proposition \ref{Prop: Serre} and the claim follows.
\end{proof}

In a similar way to \cite{Mer1}*{Corollary 5.4}, we obtain the following structural corollary by rewriting Theorem \ref{Thm: ES}, keeping the same notation as throughout this paper.

\begin{corollary}
Let $\theta$ denote the space generated by all $\Theta_{N^-}$ functions of weight $\frac{s}{2}+d^-$ for $\rho_{N^-}$. Then the equivalence classes $\Lambda_L(\psi,p,j) + M_{2j+2-k,L}^!$ generate the $\C$-vector space
\begin{align*}
\left[ \Mc^{\text{mock}}_{2j+2-k,P} , \theta \right]_j^L \slash  M_{2j+2-k,L}^!.
\end{align*}
\end{corollary}

\section{The action of the Laplace--Beltrami operator}\label{Sec: LB}
In this section, we prove Theorem \ref{Cor: intro lhmf}. To this end, we compute the action of the Laplace--Beltrami operator on the lift, and show that for certain spectral parameters, we obtain a local Maa{\ss} form. We recall that the signature of $L$ is assumed to be $(2,s)$ here. Moreover, we observe that our Siegel theta function $\Theta_{L}$ and the Siegel theta function inspected by Bruinier depend in the same way on $Z$, and thus the following result applies.
\begin{proposition}[\protect{\cite{bruinier2004borcherds}*{Proposition 4.5}}]
The Siegel theta function $\Theta_{L}(\tau,Z,p)$ considered as a function on $\H \times \H_{\ell}$ satisfies the differential equation
\begin{align*}
\Omega \Theta_{L}(\tau,Z,p) v^{\frac{\ell}{2}} = -\frac{1}{2} \Delta_{k} \Theta_{L}(\tau,Z,p) v^{\frac{\ell}{2}}.
\end{align*}
\end{proposition}

Our next step is to inspect the action of $\Omega$ on our theta lift. By Lemma \ref{lem:fdecomp} it suffices to investigate
\begin{align*}
\Psi_j^{\reg} (F_{m,k-2j,\sfr},Z) = \int_{\Fc}^{\reg} \left\langle R_{k-2j}^{j}(F_{m,k-2j,\sfr})(\tau) , \overline{\Theta_L(\tau,Z,p)} \right\rangle v^{k} d\mu(\tau).
\end{align*}
Let
\begin{align*}
H(m) \coloneqq \bigcup_{\mu \in L'\slash L} \bigcup_{\substack{\lambda \in \mu+L \\ Q(\lambda) = -m}} \lambda^\perp \subseteq \Gr(L),
\end{align*}
which collects the singularities of $\Psi_j^{\reg} (F_{m,k-2j,\sfr},Z)$ as a function of $Z$. We apply the previous proposition to our theta lift, which yields a variant of \cite{bruinier2004borcherds}*{Theorem 4.6}. 
\begin{theorem}\label{Thm: eigenval}
Let $Z \in \H_{\ell} \setminus H(m)$, and $\re(\sfr) > 1-\frac{k}{2}$. Then it holds that
\begin{align*}
\Omega \Psi_j^{\reg} (F_{m,k-2j,\sfr},Z) = \left(\sfr-\frac{k}{2}\right)\left(1-\sfr-\frac{k}{2}\right) \Psi_j^{\reg} (F_{m,k-2j,\sfr},Z).
\end{align*}
\end{theorem}

\begin{proof}
First, we note that
\begin{align*}
\Omega \Psi_j^{\reg} (F_{m,k-2j,\sfr},Z) = \int_{\Fc}^{\reg} \left\langle R_{k-2j}^{j}(F_{m,k-2j,\sfr})(\tau) , \Omega \overline{\Theta_L(\tau,Z,p)}v^{\frac{\ell}{2}} \right\rangle v^{k-\frac{\ell}{2}} d\mu(\tau),
\end{align*}
because all partial derivatives with respect to $Z$ converge locally uniformly in $Z$ as $T \to \infty$ (see \cite{bruinier2004borcherds}*{p.\ $99$}). By the previous proposition, we infer that
\begin{align*}
\Omega \Psi_j^{\reg} (F_{m,k-2j,\sfr},Z) = -\frac{1}{2} \int_{\Fc}^{\reg} \left\langle R_{k-2j}^{j}(F_{m,k-2j,\sfr})(\tau) , \Delta_{k} \overline{\Theta_L(\tau,Z,p)}v^{\frac{\ell}{2}} \right\rangle v^{k-\frac{\ell}{2}} d\mu(\tau).
\end{align*}
By the splitting $\Delta_{k} = R_{k-2}L_k$, and the adjointness of both operators (see \cite{bruinier2004borcherds}*{Lemmas 4.2 to 4.4}), we obtain
\begin{align*}
\Omega \Psi_j^{\reg} (F_{m,k-2j,\sfr},Z) = -\frac{1}{2} \int_{\Fc}^{\reg} \left\langle \Delta_{k} R_{k-2j}^{j}(F_{m,k-2j,\sfr})(\tau) , \overline{\Theta_L(\tau,Z,p)}v^{\frac{\ell}{2}} \right\rangle v^{k-\frac{\ell}{2}} d\mu(\tau).
\end{align*}
Lastly, we observe that $\Delta_{k}$ and $R_{k-2j}^j$ commute by virtue of Lemma \ref{lem:mpraise}. Namely, we have
\begin{align*}
\Delta_{k} R_{k-2j}^{j}(F_{m,k-2j,\sfr})(\tau) 
&= \left(\sfr-\frac{k}{2}\right)\left(1-\sfr-\frac{k}{2}\right) R_{k-2j}^{j}(F_{m,k-2j,\sfr})(\tau),
\end{align*}
and this establishes the claim by rewriting 
\begin{align*}
\left\langle R_{k-2j}^{j}(F_{m,k-2j,\sfr})(\tau) , \overline{\Theta_L(\tau,Z,p)}v^{\frac{\ell}{2}} \right\rangle v^{k-\frac{\ell}{2}} = \left\langle R_{k-2j}^{j}(F_{m,k-2j,\sfr})(\tau) , \overline{\Theta_L(\tau,Z,p)} \right\rangle v^{k}
\end{align*}
again.
\end{proof}
We end this section by proving Theorem \ref{Thm: intro lhmf}.

\begin{proof}[Proof of Theorem \ref{Thm: intro lhmf}]
By Theorem \ref{Thm: eigenval}, the lift is an eigenfunction of the Laplace--Beltrami operator with the quoted eigenvalue.  Since $\Psi_j^{\reg} (F_{m,k-2j,\sfr},Z)$ is an eigenfunction of an elliptic differential operator, it is real-analytic in $\Gr(L)$ outside of $H(m)$. The other conditions for the lift to be a vector-valued local Maa{\ss} form can be easily seen by applying the proof of \cite{BKV}*{Theorem 1.1} {\it mutatis mutandis}. When $\sfr = \frac{k}{2}$ or $\sfr = \frac{k}{2}-1$ we obtain locally harmonic Maa{\ss} forms.
\end{proof}

\section{Cohen--Eisenstein series} \label{Sec: CES}
We specialize the framework from Section \ref{Sec: prelims} following \cite{BS}*{Section 4.4} (or \cite{schw18}*{Section 2.2}). We fix the signature $(1,2)$ as mentioned in the introduction, and the rational quadratic space
\begin{align*}
V \coloneqq \left\{X=\left(\begin{smallmatrix} x_2 & x_1 \\ x_3 & -x_2 \end{smallmatrix}\right) \in \Q^{2\times 2}\right\},
\end{align*}
with quadratic form $Q(X) = \det(X)$. The Grassmannian of positive lines in $V \otimes \R$ can be identified with $\H$ via
\begin{align*}
\lambda(x+iy) = \frac{1}{\sqrt{2}y}\left(\begin{smallmatrix} -x & x^2+y^2 \\ -1 & x \end{smallmatrix}\right).
\end{align*}
We choose the lattice 
\begin{align*}
L \coloneqq \left\{\left(\begin{smallmatrix} b & c \\ -a & -b \end{smallmatrix}\right) \colon a, b, c \in \Z\right\},
\end{align*}
with dual lattice
\begin{align*}
L' = \left\{\left(\begin{smallmatrix} \frac{b}{2} & c \\ -a & -\frac{b}{2} \end{smallmatrix}\right) \colon a, b, c \in \Z\right\}.
\end{align*}
We observe that $L'$ can be identified with the set of integral binary quadratic forms of discriminant $\det\left(\begin{smallmatrix} \frac{b}{2} & c \\ -a & -\frac{b}{2} \end{smallmatrix}\right) = -\frac{1}{4}(b^2-4ac)$. Furthermore, $L'\slash L \cong \Z\slash 2\Z$ with quadratic form $x \mapsto -\frac{1}{4}x^2$.

According to \cite{BS}*{p.\ 22}, it holds that
\begin{align*}
Q\left(\left(\begin{smallmatrix} \frac{b}{2} & c \\ -a & -\frac{b}{2} \end{smallmatrix}\right)_{x+iy}\right) &= \frac{1}{4y^2}\left(a(x^2+y^2)+bx+c\right)^2, \\
Q\left(\left(\begin{smallmatrix} \frac{b}{2} & c \\ -a & -\frac{b}{2} \end{smallmatrix}\right)_{{(x+iy)}^\perp}\right) &= -\frac{1}{4y^2}|[a,b,c](x+iy,1)|^2.
\end{align*}
We remark that both are invariant under modular substutions. By a result from Eichler and Zagier \cite{EZ}*{Theorem 5.4}, the space of vector-valued modular forms of weight $k$ for $\rho_{L}$ is isomorphic to the space $M_{k}^+(\Gamma_0(4))$ of scalar-valued modular forms satisfying the Kohnen plus space condition via the map
\begin{align*}
f_0(\tau)\e_0 + f_1(\tau)\e_1 \mapsto f_0(4\tau)+f_1(4\tau).
\end{align*}
This enables us to use scalar-valued forms as inputs for our theta lift. 

\begin{proof}[Proof of Theorem \ref{Cor: intro lhmf}]
As outlined in the introduction, the function $f(\tau) \coloneqq f_{-2\ell,N}(\tau)\Hc_{\ell}(\tau)$ is of weight $-\ell-\frac{1}{2} < 0$ for $\Gamma_0(4)$, has non-constant principal part at the cusp $i\infty$, and its image under $\xi$ is trivial, hence in particular cuspidal. This enables us to apply Corollary \ref{cor: theta lift as 2F1} to $f$. To this end, we have the parameters
\begin{align*}
k= -\frac{1}{2}+d^++d^-, \qquad k-2j = -\ell-\frac{1}{2}.
\end{align*}
Rewriting those yields
\begin{align*}
j = \frac{\ell+d^++d^-}{2},
\end{align*}
and the hypergeometric function from Theorem \ref{Theorem: theta lift of MPseries} becomes
\begin{align*}
{}_2F_1\left(\frac{\ell+2+d^++d^-}{2},\frac{\ell+2+d^++3d^-}{2},\frac{5}{2}+\ell,\frac{4my^2}{|[a,b,c](z,1)|^2}\right).
\end{align*}

Inspecting the parameters, we have the condition $\ell+d^++d^- \in 2\N$ by $j \in \N$, and combining with $d^+$, $d^- \in \N_0$, $\ell \in \N\setminus\{1\}$, the smallest possible values are $(\ell,d^+,d^-) = (2,0,0)$, $(2,2,0)$, $(2,1,1)$, $(2,0,2)$. For example, the corresponding hypergeometric functions for the cases $(\ell,d^+,d^-) = (2,0,0)$, $(2,1,1)$ are
\begin{align*}
{}_2F_1\left(2,2,\frac{9}{2},\tilde{z}\right) &= -\frac{35(11\tilde{z}-15)}{12\tilde{z}^3} - \frac{35(2\tilde{z}^2-7\tilde{z}+5)\arcsin(\sqrt{\tilde{z}})}{4\tilde{z}^{\frac{7}{2}}\sqrt{1-\tilde{z}}}, \\
{}_2F_1\left(3,4,\frac{9}{2},\tilde{z}\right) &= -\frac{35(8\tilde{z}^2-26\tilde{z}+15)}{128\tilde{z}^3(\tilde{z}-1)^2} + \frac{105(8\tilde{z}^2-12\tilde{z}+5)\arcsin(\sqrt{\tilde{z}})}{128\tilde{z}^{\frac{7}{2}}\sqrt{1-\tilde{z}}(\tilde{z}-1)^2},
\end{align*}
and the other cases are of similar shape. Analogous expressions can be obtained for higher integer parameters via Gau{\ss}' contiguous relations for the hypergeometric function, which can be found in \cite{nist}*{§15.5(ii)} for instance.

We infer a local behaviour as sketched in the introduction by virtue of ($4m = D = b^2-4ac$)
\begin{align*}
\arcsin\left(\sqrt{\tilde{z}}\right) = \arcsin\left(\frac{\sqrt{D}y}{\vt{az^2+bz+c}}\right) = \arctan\vt{\frac{\sqrt{D}y}{a\vt{z}^2+bx+c}},
\end{align*}
which in turn follows by
\begin{align*}
(b^2-4ac)y^2 + \left(a\vt{z}^2+bx+c\right)^2  = \vt{az^2+bz+c}^2,
\end{align*}
compare \cite{BKK}*{Section 3}. The denominator $a\vt{z}^2+bx+c$ vanishes if and only if $z$ is located on the Heegner geodesic associated to $Q=[a,b,c]$. Since the principal part of $f$ is given by
\begin{align*}
\sum_{n=0}^N H(\ell,n)q^{n-N} + O\left(q^{m+1}\right), \qquad m = \begin{cases}
\lfloor \frac{-2\ell}{12}\rfloor-1 & \text{if } -2\ell \equiv 2\pmod*{12}, \\
\lfloor \frac{-2\ell}{12}\rfloor & \text{else},
\end{cases}
\end{align*}
we conclude that $f$ has the exceptional set
\begin{align*}
\bigcup_{D = 1}^N \left\{z = x+iy \in \H \colon \exists a,b,c \in \Z, \ b^2-4ac=D, \ a\vt{z}^2+bx+c = 0\right\}.
\end{align*}
In other words, the exceptional set of $f$ is a finite union of nets of Heegner geodesics. Furthermore, we recall that the spectral parameter in Corollary \ref{cor: theta lift as 2F1} is $\sfr = 1-\frac{k-2j}{2}$, and hence the eigenvalue under $\Delta_{-\ell-\frac{1}{2}}$ is
\begin{align*}
\left(\sfr-\frac{k}{2}\right)\left(1-\sfr-\frac{k}{2}\right) = \left(1-k+j\right)(-j) = j\left(j-\ell-\frac{3}{2}\right).
\end{align*}
This proves the Theorem.
\end{proof}

\subsection{Eichler--Selberg type relations for Cohen--Eisenstein series}
Eichler--Selberg type relations for Cohen--Eisenstein series could be obtained as follows. On one hand, the input function $f(\tau) = f_{-2\ell,N}(\tau)\Hc_{\ell}(\tau)$ is weakly holomorphic, thus we do not need to deal with the additional term 
\begin{align*}
\int_{\Fc}^{\reg} \left\langle L_{k-2j} \left(f_{P\oplus N}\right)(\tau), \left[\Gc_P^+(\tau),\Theta_{N^-}(\tau)\right]_j  \right\rangle v^{-2} d\tau
\end{align*}
arising from Theorem \ref{Theorem: theta lift at CM points}. On the other hand, the function $\Lambda_L$ from Section \ref{sec:ESrelations} simplifies to
\begin{multline*}
\Lambda_L(\psi,p,j) = \frac{4^{3d^-}\pi^{\frac{1}{2}-d^+} \Gamma\left(j+1+d^-\right)\Gamma\left(\frac{3}{2}-d^++j\right)}{\Gamma\left(\ell+\frac{1}{2}\right)\Gamma\left(\frac{3}{2}-d^+\right)} \sum_{D \geq 1} \sum_{Q \in \Qc_D} \overline{p(\psi(Q))} \\ 
\times \frac{D^{\ell+\frac{3}{2}}y^{2+2j+2d^-}}{\vt{Q(z,1)}^{2+2j+2d^-}}
{}_2F_1\left(\frac{\ell+2+d^++d^-}{2},\frac{\ell+2+d^++3d^-}{2},\frac{5}{2}+\ell,\frac{Dy^2}{\vt{Q(z,1)}^2}\right) q^D,
\end{multline*}
where $\Qc_D$ denotes the set of integral binary quadratic forms of discriminant $D$. After evaluating the hypergeometric function as in the previous proof, one may follow our proof of Theorem \ref{Thm: ES}, namely subtract the two evaluations of $\Psi^{\reg}_j \left(f, z\right)$ from each other, and apply Serre duality to the resulting expression. Computing the principal part of $\Gc_P^+$ in addition, this yields the desired result. However, we do not pursue this here explicitly as the resulting expression is rather lengthy.

\end{document}